\documentclass[a4paper,reqno,11pt]{amsart}
\usepackage{amsmath}
\usepackage{amssymb}
\usepackage{cases}
\usepackage{enumitem}
\usepackage[all]{xy}
\usepackage{quiver,derivative}
 
\allowdisplaybreaks[4]

\newtheorem{thm}{Theorem}[section]
\newtheorem{prop}[thm]{Proposition}
\newtheorem{prop-def}[thm]{Proposition-Definition}

\newtheorem{lem}[thm]{Lemma}

\newtheorem{claim}[thm]{Claim}

\theoremstyle{definition}
\newtheorem{definition}[thm]{Definition}

\newtheorem{notation}[thm]{Notation}

\theoremstyle{remark}
\newtheorem{remark}[thm]{Remark}

\numberwithin{equation}{section}

\newcommand{\bQ}{\mathbb{Q}}
\newcommand{\bR}{\mathbb{R}}

\newcommand{\bP}{\mathbb{P}}

\newcommand\OO{{\mathcal{O}}}

\newcommand{\Vol}{\operatorname{Vol}}
\newcommand{\irr}{\operatorname{irr}}

\usepackage{todonotes}

\begin{document}

\title{A canonical Fano threefold has degree $\leq 72$}
\dedicatory{Dedicated to Professor Meng Chen on the occasion of his 60th birthday}
\date{\today}
\author{Chen Jiang}
\address{Chen Jiang, Shanghai Center for Mathematical Sciences \& School of Mathematical Sciences, Fudan University, Shanghai 200438, China}
\email{chenjiang@fudan.edu.cn}

\author{Tianqi Zhang}
\address{Tianqi Zhang, Shanghai Center for Mathematical Sciences, Fudan University, Shanghai 200438, China}
\email{tqzhang22@m.fudan.edu.cn}

\author{Yu Zou}
\address{Yu Zou, College of Mathematics and Statistics, Chongqing University \& Key Laboratory of Nonlinear Analysis and its Applications (Chongqing University), Ministry of Education, Chongqing 401331, China}
\email{fishlinazy@cqu.edu.cn}


\begin{abstract}
We show that the anti-canonical volume of a canonical weak Fano $3$-fold is at most $72$. This upper bound is optimal. 
\end{abstract}

\maketitle
\pagestyle{myheadings} \markboth{\hfill C.~Jiang, T.~Zhang, Y.~Zou
\hfill}{\hfill A canonical Fano threefold has degree $\leq 72$\hfill}


\section{Introduction}
Throughout this paper, we work over the field of complex numbers $\mathbb{C}$.

A normal projective variety $X$ is called a {\it
Fano variety} (resp., {\it weak Fano variety}) if the anti-canonical divisor $-K_X$ is ample (resp., nef and big). 

According to the minimal model program, Fano varieties are one of the most important research objects in birational geometry. As canonical Fano $3$-folds form a bounded family (see \cite{Kaw92, KMMT00,Bir16b}), different kinds of invariants of canonical Fano $3$-folds are bounded. The {\it anti-canonical volume} (also called {\it degree}) of a Fano $3$-fold is one of the most important invariants and it plays an important role in the classification of smooth Fano $3$-folds (see \cite{IP99}).
Motivated by the classification theory of canonical Fano $3$-folds, we aim to find the optimal upper bound for their anti-canonical volumes.

It was conjectured by Prokhorov (after Fano and Iskovskikh) that the optimal upper bound for the anti-canonical volume of a canonical Fano $3$-fold should be $72$. This conjecture was confirmed by Prokhorov for Gorenstein canonical Fano $3$-folds and for $\bQ$-factorial terminal Fano $3$-folds of Picard number $1$ (\cite{Prok05, Prok07}). 
Later, the first author and the third author gave the first explicit upper bound $324$ for anti-canonical volumes of canonical Fano $3$-folds (\cite{JZ23}). Recently, the first author, Haidong Liu, and Jie Liu proved this conjecture for $\mathbb{Q}$-factorial canonical Fano $3$-folds of Picard number $1$ (\cite{JLL25}). See \cite{JZ23} and the reference therein for related results in this direction.

As the main result of this paper, we show that $72$ is the optimal upper bound for anti-canonical volumes of all canonical weak Fano $3$-folds and confirm Prokhorov's conjecture. The equality case is also characterized. 
\begin{thm}\label{mainthm}
Let $X$ be a canonical weak Fano $3$-fold. Then the following assertions hold:
\begin{enumerate}
 \item $(-K_{X})^3\leq 72$;
 \item for any positive integer $m$,
 $
h^0(X, -mK_X)\leq (2m+1)(6m^2+6m+1)$;
 
 \item one of the above inequalities is an equality if and only if there is a birational morphism $g: X\to X_0$ where $X_0\simeq \mathbb{P}(1,1,1,3)$ or $\mathbb{P}(1,1,4,6)$ and $g^*K_{X_0}=K_X$; in particular, if $X$ is a canonical Fano $3$-fold, then one of the above inequalities is an equality if and only if $X\simeq X_0$.
\end{enumerate}
\end{thm}

\begin{remark}\label{rem vol>64}
\begin{enumerate}
   \item One can use toric geometry to describe more explicitly those canonical weak Fano $3$-folds in Theorem~\ref{mainthm}(3) which are just crepant extractions of $X_0$.
For example, if $X_0\simeq \mathbb{P}(1,1,1,3)$, 
then either $X\simeq X_0$ or $X\simeq \mathbb{P}_{\mathbb{P}^2}(\mathcal{O}_{\mathbb{P}^2}\oplus \mathcal{O}_{\mathbb{P}^2}(3))$; for the computation related to $X_0\simeq \mathbb{P}(1,1,4,6)$, we omit the detail and refer to \cite[Remark~6.2]{Lai21}.

    \item In the proof of Theorem~\ref{mainthm}, it is shown that if $X$ is a canonical weak Fano $3$-fold with $(-K_X)^3>64$, then there exists a birational contraction $X'\dashrightarrow X_0$ such that $X_0$ is a $\mathbb{Q}$-factorial canonical Fano $3$-fold with $\rho(X_0)=1$, where $X'$ is a $\mathbb{Q}$-factorial terminalization of $X$. 
\end{enumerate}
 
\end{remark} 

We briefly explain the ideas in the proof of Theorem~\ref{mainthm}.
By using the minimal model program, we can construct two birational models of $X$: one is a terminal $3$-fold $Y$ with a fibration structure $Y\to S$ and the other is a $\mathbb{Q}$-factorial canonical weak Fano $3$-fold $Z$, where $\rho(Y)=\rho(Z)$. Then the problem is reduced to finding the upper bound for $(-K_Z)^3=\Vol(Y, -K_Y)$.
This construction combines ideas in \cite{CJ20, JZ23}.

We may assume that $Z$ has no birational contraction to a $\mathbb{Q}$-factorial canonical Fano $3$-fold with $\rho=1$, otherwise the problem is solved due to \cite{JLL25}. This recent development allows us to work with the model $Z$ (with canonical singularities) instead of being restricted to the model $Y$ (with terminal singularities) as in \cite{JZ23}. 

Consider the case $\rho(Y)=\rho(Z)=2$ first. 
We can partly use the computation of Lai in \cite{Lai21} (see also \cite{LaiLee}), where Lai treated $\mathbb{Q}$-factorial terminal Fano $3$-folds with $\rho=2$ by studying Mori fiber spaces of $Z$. The most difficult part here is when both Mori fiber spaces of $Z$ are conic bundles, 
whose bases are del Pezzo surfaces which might have worse than Du Val singularities. The lack of explicit classification of such surfaces makes the computation in \cite{Lai21} impossible for $Z$.
One advantage
of our construction is that $Y$ is one of the Mori fiber spaces of $Z$ and is terminal, so we can work on $Y$ instead of $Z$. But here $Y$ is no longer a weak Fano $3$-fold and we can not easily bound the volume using intersection numbers. 

Following the ideas in \cite{Jiang18, Jiang21}, to bound the anti-canonical volume of $Y$ from above, it suffices to give an upper bound for the number $w$ where $-K_Y-wG$ is big for a well-chosen free divisor $G$ on $Y$. In fact, we can show that $w<2$ if both Mori fiber spaces of $Z$ are conic bundles. 
This idea works for more general case 
$\rho(Y)=\rho(Z)\geq 3$ as well by using the concept of $\mathbb{Q}$-fixed prime divisor (see Definition~\ref{fixed divisor}): if $Z$ has a $\mathbb{Q}$-fixed prime divisor, then we can contract this divisor to reduce to smaller Picard number case; otherwise, $Z$ has a Mori fiber space which is a conic bundle and gives a good upper bound for $w$. 

Finally, the upper bound of $w$ reduces our problem to a detailed calculation of volumes of divisors on the surface $G$ which eventually gives the optimal upper bound for $(-K_Z)^3$.

This paper is organized as follows. In Section~\ref{sec 2}, we introduce definitions and basic knowledge. In Section~\ref{sec 3}, we study birational models of weak Fano varieties and construct $2$ birational models of a $\mathbb{Q}$-factorial terminal weak Fano $3$-fold (Proposition~\ref{prop MFS}). In Section~\ref{sec 4}, we study the geometry of $\mathbb{Q}$-factorial canonical weak Fano $3$-folds with $\rho\geq 2$ 
(Theorem~\ref{thm picard 23}). In Section~\ref{sec 5}, we prove the main theorem.

\section{Preliminaries}\label{sec 2}
We adopt standard notation and definitions in \cite{KM} and will freely use them.


\subsection{Singularities}
\begin{definition}
A {\it pair} $(X, B)$ consists of a normal variety $X$ and an effective
$\bR$-divisor $B$ on $X$ such that
$K_X+B$ is $\bR$-Cartier.
\end{definition}

\begin{definition}\label{def sing}
Let $(X, B)$ be a pair. Let $f: Y\to X$ be a log
resolution of $(X, B)$, write
\[
K_Y =f^*(K_X+B)+\sum a_iE_i,
\]
where $E_i$ are distinct prime divisors on $Y$ satisfying $f_*(\sum a_iE_i)=-B$. 
The coefficient $a_i$ is called the {\it discrepancy} of $E_i$ with respect to $(X, B)$, and is denoted by $a(E_i, X, B)$. 
The pair $(X,B)$ is called
\begin{enumerate}
\item \emph{Kawamata log terminal} ({\it klt},
for short) if $a_i>-1$ for all $i$;

\item \emph{terminal} if $a_i> 0$ for all $f$-exceptional divisors $E_i$ and for all $f$;

\item \emph{canonical} if $a_i\geq 0$ for all $f$-exceptional divisors $E_i$ and for all $f$.
\end{enumerate}
Usually, we say that $X$ is {\it terminal} (resp., {\it canonical}, {\it klt}) if $(X,0)$ is terminal (resp., canonical, klt).

\end{definition}

\subsection{Volumes}
\begin{definition}
Let $X$ be an $n$-dimensional projective variety and $D$ be a Cartier divisor on $X$. The {\it volume} of $D$ is defined by
\[
\Vol(X, D)=\limsup_{m\to \infty}\frac{h^0(X, \OO_X(mD))}{m^n/n!}.
\]
For more details and properties of volumes, we refer to \cite[2.2.C]{Positivity1} and \cite[11.4.A]{Positivity2}. Moreover, by homogeneous property and continuity of volumes, the definition can be extended to $\bR$-Cartier $\bR$-divisors. 
An $\bR$-Cartier $\bR$-divisor $D$ on $X$ is big if and only if $\Vol(X, D)>0$. 
Note that if $D$ is nef, then $\Vol(X, D)=D^n$. 
\end{definition}

The following lemma is very useful for estimating volumes by induction on dimensions. 

\begin{lem}\label{lem vol inequality}
 Let $X$ be a smooth projective variety. Let $D$ be a big $\mathbb{R}$-Cartier $\mathbb{R}$-divisor and let $E$ be a prime divisor on $X$. Then for any positive real number $x$,
 \[
 \Vol(X, D)-\Vol(X, D-xE)\leq \dim X\cdot \int_{0}^x\Vol(E, (D-tE)|_E) \,\mathrm{d}t.
 \]
 
\end{lem}
\begin{proof}
 Denote
 $\tau:=\max\{s>0\mid D-sE \text{ is big}\}.$

If $x<\tau$, 
then by \cite[Corollary~C]{BFJ09} or \cite[Corollary~C]{LM09}, 
 \[
 \Vol(X, D)-\Vol(X, D-xE)= \dim X\cdot \int_{0}^x\Vol_{X|E}(D-tE) \,\mathrm{d}t.
 \]
 Here $\Vol_{X|E}(D-tE)$ is the {\it restricted volume} (see \cite[\S\,4.2]{BFJ09}) and by definition it is clear that \[\Vol_{X|E}(D-tE)\leq \Vol(E, (D-tE)|_E).\]
Hence we conclude the desired inequality. Moreover, the inequality also holds for $x=\tau$ by continuity. 

If $x> \tau$, then $\Vol(X, D-xE)=0$ and 
\begin{align*}
 \Vol(X, D)\leq {}&\dim X\cdot \int_{0}^\tau\Vol(E, (D-tE)|_E) \,\mathrm{d}t\\
 \leq {}&\dim X\cdot \int_{0}^x\Vol(E, (D-tE)|_E) \,\mathrm{d}t. \qedhere
\end{align*}
\end{proof}

\subsection{Ruled surfaces and weak del Pezzo surfaces}
In this subsection, we recall some basic properties of ruled surfaces and weak del Pezzo surfaces.

First, we recall the computation of volumes of divisors on Hirzebruch surfaces or ruled surfaces over an elliptic curve by the Zariski decomposition (see \cite[Theorem~2.3.19]{Positivity1}).
We will use the following notation. 
\begin{notation}\label{notation Fn En}
 \begin{enumerate}
 \item For $n\geq 0$, denote by $\mathbb{F}_n$ the $n$-th Hirzebruch surface
$\mathbb{P}_{\mathbb{P}^1}(\mathcal{O}_{\mathbb{P}^1}\oplus\mathcal{O}_{
\mathbb{P}^1}(-n))$. 
If $n\geq 1$, denote by $\sigma_0$ and $\ell$ the unique negative section and the ruling of $\mathbb{F}_n$. For consistency of the notation, 
 denote also by $\sigma_0$ and $\ell$ the two rulings of $\mathbb{F}_0$.

\item For $n\geq -1$, denote by $\mathbb{E}_n$ a ruled surface over an elliptic curve with invariant $n$ (see \cite[Notation~V.2.8.1]{H}).
Denote by $\sigma_0$ the section with $\sigma_0^2=-n$ and by $\ell$ the ruling of $\mathbb{E}_n$.
 \end{enumerate}
\end{notation}

\begin{lem}\label{lem Fn En divisors}
Let $S=\mathbb{F}_n$ or $\mathbb{E}_n$ for $n\geq 0$.
Then the following assertions hold.
\begin{enumerate}
 \item The cone of effective divisors of $S$ is generated by $\sigma_0$ and $\ell$, and the cone of nef divisors of $S$ is generated by $\sigma_0+n\ell$ and $\ell$.

 \item 
 For real numbers $s$ and $t$, 
\[
\Vol(S, s\sigma_0+t\ell)=\begin{cases}
 \frac{t^2}{n} & \text{if } ns>t\geq 0;\\
 2st-ns^2 & \text{if } t\geq ns\geq 0 \text{ and }s \geq 0;\\
 0 & \text{otherwise}.
\end{cases}
\]
\end{enumerate}

\end{lem}

\begin{proof} 
(1) This follows from \cite[Proposition~2.20]{H}.

(2) We may assume that $s\geq 0$ and $t\geq 0$; otherwise $\Vol(S, s\sigma_0+t\ell)=0$.

If $ t\geq ns$, then $s\sigma_0+t\ell$ is nef, so \[\Vol(S,s\sigma_0+t\ell)=(s\sigma_0+t\ell)^2=2st-ns^2.\]

 If $ ns>t$ (in which case $n\neq 0$), then the Zariski decomposition of $s\sigma_0+t\ell$ is
\[s\sigma_0+t\ell=\frac{t}{n}(\sigma_0+n\ell)+(s-\frac{t}{n})\sigma_0,\]
and hence \[\Vol(S,s\sigma_0+t\ell)=\frac{t^2}{n^2}(\sigma_0+n\ell)^2=\frac{t^2}{n}.\qedhere\]
\end{proof}

\begin{lem}\label{lem E-1 divisors}
\begin{enumerate}
 \item The cone of effective divisors of $\mathbb{E}_{-1}$ and the cone of nef divisors of $\mathbb{E}_{-1}$ coincide and are generated by $2\sigma_0- \ell$ and $\ell$.

 \item 
 For real numbers $s$ and $t$, 
\[
\Vol(\mathbb{E}_{-1}, s\sigma_0+t\ell)=\begin{cases} 
 s(s+2t) & \text{if } s\geq 0 \text{ and } s+2t\geq 0;\\
 0 & \text{otherwise}.
\end{cases}
\]
\end{enumerate}

\end{lem}

\begin{proof} 
(1) This follows from \cite[Proposition~2.21]{H}.

(2) We may assume that $s\geq 0$ and $s+2t\geq 0$; otherwise $\Vol(\mathbb{E}_{-1}, s\sigma_0+t\ell)=0$.
 Then $s\sigma_0+t\ell$ is nef and \[\Vol(\mathbb{E}_{-1},s\sigma_0+t\ell)=(s\sigma_0+t\ell)^2=s^2+2st.\qedhere\]
\end{proof}

 A (weak) Fano surface is usually called a {\it (weak) del Pezzo surface}. The following lemma is well-known to experts.

\begin{lem}\label{lem DP surface}
 Let $S$ be a weak del Pezzo surface with at worst Du Val singularities. Then the following assertions hold:
 \begin{enumerate}
 \item $K_S^2\in \{1,2,\dots, 9\}$;
\item if $-K_S\sim_{\mathbb{Q}}aD+B$ where $D$ is a movable Weil divisor and $B$ is an effective $\mathbb{Q}$-divisor on $S$, then
\[
a\leq \begin{cases} 4& \text{if } K_S^2\leq 8;\\
3& \text{if } K_S^2=9.
\end{cases}
\]
 \end{enumerate}
\end{lem}
\begin{proof}
 Let $\pi:S'\to S$ be the minimal resolution of $S$. Then $\pi^*K_{S}=K_{S'}$ and $S'$ is a smooth weak del Pezzo surface. By \cite[Theorem~8.1.15]{Dol12}, $S'$ has a birational morphsim to one of $\bP^2$, $\mathbb{F}_0$, $\mathbb{F}_2$. Then it follows that $K_S^2=K_{S'}^2\in\{1,2,\dots, 9\}.$

 If $-K_S\sim_{\mathbb{Q}}aD+B$ where $D$ is a movable Weil divisor and $B$ is an effective $\mathbb{Q}$-divisor on $S$, then $-K_{S'}\sim_{\mathbb{Q}}aD'+B'$ where $D'\in\text{Mov}|\lfloor \pi^*D\rfloor|$ is a movable Weil divisor and $B'$ is an effective $\mathbb{Q}$-divisor on $S'$. By \cite[Lemma~4.4]{JZ24}, $a\leq 4$. Moreover, if $K_{S}^2=9$, then $S\simeq S'\simeq \mathbb{P}^2$ and $a\leq 3$ by a degree computation. 
\end{proof}
\section{Birational models of weak Fano varieties}\label{sec 3}

In this section, we study birational models of weak Fano varieties. The main goal is to construct $2$ special birational models for a $\mathbb{Q}$-factorial terminal weak Fano $3$-fold (Proposition~\ref{prop MFS}). 

A {\it birational contraction} is a birational map which is surjective in codimension $1$. For birational contractions, we will often use the following well-known lemma to compare divisors.

 \begin{lem}\label{lem bir cont preserve sections}
 Let $f: X\dashrightarrow X'$ be a birational contraction between normal projective varieties and let $D$ be a Weil divisor on $X$.
 Then
 \[
 h^0(X', f_*D)\geq h^0(X, D).
 \]
 In particular, if $D$ is big, then $f_*D$ is big. 
 \end{lem}
 \begin{proof}
 Since $f$ is a birational contraction, $f_*$ gives a bijection between principal divisors and preserves linear equivalence. Hence $f_*$ induces an injective map from $|D|$ to $|f_*D|$, which shows the desired inequality. Applying the inequality for $mD$ for any positive integer $m$, we get the last assertion.
 \end{proof}

\subsection{Mori fiber spaces, varieties of Fano type, and Mori dream spaces}
We recall the concepts of Mori fiber spaces, varieties of Fano type, and Mori dream spaces. 
\begin{definition}
Let $f: Y\to T$ be a surjective morphism between normal projective varieties and let $D$ be an $\mathbb{R}$-divisor on $Y$. Then $f: Y\to T$ is called a {\it $D$-Mori fiber space} if the following conditions are satisfied:
 \begin{enumerate}
 \item $Y$ is $\mathbb{Q}$-factorial; 
 \item $-D$ is ample over $T$;
 
 \item $f_*\OO_Y=\OO_T$;
 
 \item $\rho(Y/T)=1$;
 \item $\dim Y>\dim T$.
 \end{enumerate}
A $K_Y$-Mori fiber space is just called a {\it Mori fiber space}. 
\end{definition}
\begin{remark}
 If $f:Y\to T$ is a $D$-Mori fiber space, then it is a $D'$-Mori fiber space for any $\mathbb{R}$-divisor $D'$ on $Y$ such that $-D'$ is big over $T$ as $\rho(Y/T)=1$. In particular, if $-K_Y$ is big over $T$, then every $D$-Mori fiber space is automatically a Mori fiber space. 
\end{remark}
 
\begin{remark}
 Let $Y\to T$ be a Mori fiber space where $Y$ is a canonical $3$-fold. Then there are three types of Mori fiber spaces:
 \begin{enumerate}
 \item $\dim T=0$ and $Y$ is a $\mathbb{Q}$-factorial canonical Fano $3$-fold with $\rho(Y)=1$;
 
 \item $\dim T=1$ and a general fiber of $Y\to T$ is a del Pezzo surface with Du Val singularities; sometimes $Y\to T$ is called a {\it del Pezzo fibration};

 \item $\dim T=2$ and a general fiber of $Y\to T$ is a smooth rational curve; sometimes $Y\to T$ is called a {\it conic bundle}. 
 \end{enumerate}
\end{remark}
 
 \begin{definition}[{\cite{PS09}}]
A projective variety $X$ is said to be {\it of Fano type} if one of the following equivalent conditions is satisfied:
\begin{enumerate}
 \item there exists a klt pair $(X,B)$ such that $-(K_X+B)$ is ample;
 \item there exists a klt pair $(X,B')$ such that $-(K_X+B')$ is nef and big;
 \item there exists a klt pair $(X,B'')$ such that $K_X+B''\equiv 0$ and $-K_X$ is big. 
\end{enumerate} 
\end{definition}
\begin{remark}\label{fano type}
 
If $X$ is $\mathbb{Q}$-factorial and
of Fano type, then it is a Mori dream space (see \cite[Corollary~1.3.2]{BCHM} and \cite{HK00});
roughly speaking, we can run $D$-MMP for any $\mathbb{R}$-divisor $D$ on $X$ which terminates to a good $D$-minimal model or a $D$-Mori fiber space, depending on whether $D$ is pseudo-effective or not.
 
\end{remark}

\subsection{Properties of birational models of weak Fano varieties}
 In this subsection, we study properties of birational models of weak Fano varieties. 

Recall the following definitions for divisors. 
\begin{definition}\label{fixed divisor}
 Let $X$ be a normal projective variety and let $D$ be a non-zero Weil divisor on $X$. 

\begin{enumerate}
\item $D$ is said to be {\it movable} if 
 the base locus $\textrm{Bs}|D|$ has codimension at least $2$;
 
 \item $D$ is said to be {\it $\mathbb{Q}$-movable} if 
there exists a positive integer $m$ such that $mD$ is movable;

\item $D$ is said to be a {\it $\mathbb{Q}$-fixed prime divisor} if $D$ is a prime divisor and is not $\mathbb{Q}$-movable, or equivalently, $h^0(X, mD)=1$ for any positive integer $m$.
\end{enumerate}
\end{definition}
We will often use the fact that a $\mathbb{Q}$-movable divisor on a surface is nef.

The following lemma shows basic properties of the image of a birational contraction from a 
$\mathbb{Q}$-factorial canonical weak Fano variety.

\begin{lem}\label{lem X to X' bir cont}
 Let $X$ be a $\mathbb{Q}$-factorial canonical weak Fano variety. Let $f: X\dashrightarrow X'$ be a birational contraction where $X'$ is $\mathbb{Q}$-factorial and projective.
 Then the following assertions hold:
 \begin{enumerate}
 \item $X'$ is canonical;
 \item $-K_{X'}$ is big and $\mathbb{Q}$-movable;
 
 \item $X'$ is of Fano type;
 
 \item for any integer $m\geq1$, \[h^0(X,-mK_X)\leq h^0({X'},-mK_{X'}) ;\]
 \item $\Vol(X , -K_{X})\leq \Vol(X', -K_{X'})$. 
 \end{enumerate} 
\end{lem}
\begin{proof}
As $X$ is canonical and $-K_X$ is nef and big, $-K_X$ is semi-ample by the basepoint-free theorem (\cite[Theorem~3.3]{KM}).
 For any sufficiently large and divisible positive integer $n$, $|-nK_X|$ is basepoint-free. Hence for a general member $M\in |-nK_X|$, $M$ is irreducible and $(X, \frac{1}{n}M)$ is canonical by Bertini's theorem (\cite[Lemma~5.17]{KM}).
Take $W$ to be a common resolution of $X$ and $X'$ with natural morphisms $p:W\to X$ and $q:W\to X'$.
Denote by $M'$ the strict transform of $M$ on $X'$, then 
\[
p^*\Big(K_X+\frac{1}{n}M\Big)=q^*\Big(K_{X'}+\frac{1}{n}M'\Big)\sim_\bQ 0
\]
by the negativity lemma (\cite[Lemma~3.39]{KM}). This implies that for any prime divisor $E$ over ${X'}$, $a(E, X, \frac{1}{n}M)=a(E, {X'}, \frac{1}{n}M')$. As $M$ is irreducible and basepoint-free, it is not exceptional over ${X'}$. So for any prime divisor $E$ which is exceptional over ${X'}$, either $E$ is exceptional over $X$, or $E$ is a prime divisor on $X$ such that $E\neq M$. In either case, $a(E, {X'}, \frac{1}{n}M')=a(E, X, \frac{1}{n}M)\geq 0$. Hence $({X'}, \frac{1}{n}M')$ is canonical. In particular, ${X'}$ is canonical.
Since $-K_X$ is big and semi-ample, $-K_{X'}=f_*(-K_X)$ is big and $\mathbb{Q}$-movable.
Since $K_{X'}+\frac{1}{n}M'\sim_\bQ 0$, $X'$ is of Fano type. 

By Lemma~\ref{lem bir cont preserve sections}, for any positive integer $m$,
\begin{align*}
 h^0(X, -mK_X) \leq h^0({X'}, -mK_{X'}). 
\end{align*}
Then by definition, $\Vol(X , -K_{X})\leq \Vol(X', -K_{X'})$. 
\end{proof}

The following lemma shows that a prime divisor contracted by a birational contraction is $\mathbb{Q}$-fixed. 

 \begin{lem}\label{lem exc is fixed}
 Let $f: X\dashrightarrow X'$ be a birational contraction between normal projective varieties and let $D$ be a prime divisor on $X$.
 If $D$ is exceptional over $X'$, then $D$ is $\mathbb{Q}$-fixed.
 \end{lem}
 \begin{proof}
 By Lemma~\ref{lem bir cont preserve sections}, for any positive integer $m$,
 \[h^0(X', mf_*D)\geq h^0(X, mD)\geq 1.\]
 This implies that $D$ is $\mathbb{Q}$-fixed as $D$ is exceptional. 
 \end{proof}

Conversely, we can contract a 
$\mathbb{Q}$-fixed prime divisor on a weak Fano variety by a birational contraction.

\begin{lem}\label{lem contract fixed}
 Let $X$ be a $\bQ$-factorial canonical weak Fano variety and $D$ be a $\bQ$-fixed prime divisor on $X$. Then there exists a birational contraction $f:X\dashrightarrow Y$ satisfying the following properties:
 \begin{enumerate}
 \item $Y$ is a $\bQ$-factorial canonical weak Fano variety;
 \item $D$ is contracted by $f$ and $\rho(X)=\rho(Y)+1$. 
 \end{enumerate}
 \begin{proof}
 Since $X$ is a Mori dream space by \cite[Corollary~1.3.2]{BCHM}, we can run a $D$-MMP on $X$ which gives a birational contraction $\phi:X\dashrightarrow X'$ such that $X'$ is $\mathbb{Q}$-factorial and $\phi_*D$ is semi-ample. 
 Since $\phi$ is obtained by a $D$-MMP and $D$ is $\mathbb{Q}$-fixed, \[h^0(X', m\phi_*D)=h^0(X, mD)=1\] for any positive integer $m$. This implies that $\phi_*D=0$. Hence $D$ is contracted by $\phi$ and is the only $\phi$-exceptional divisor on $X$. In particular, $\rho(X)=\rho(X')+1$.

 By Lemma~\ref{lem X to X' bir cont}, $X'$ is of Fano type, hence it is a Mori dream space by \cite[Corollary~1.3.2]{BCHM}. 
 Since $-K_{X'}$ is big, we may run a $(-K_{X'})$-MMP on $X'$ which ends up with a variety $Y$ such that $Y$ is $\mathbb{Q}$-factorial and $-K_Y$ is nef and big. Denote by $f: X\dashrightarrow Y$ the natural birational contraction.
By Lemma~\ref{lem X to X' bir cont} for $X\dashrightarrow Y$, $Y$ is canonical, hence it is a $\bQ$-factorial canonical weak Fano variety.
Since $-K_{X'}$ is $\mathbb{Q}$-movable by Lemma~\ref{lem X to X' bir cont} for $X\dashrightarrow X'$, the $(-K_{X'})$-MMP on $X'$ does not contract any divisor. So $X'\dashrightarrow Y$ is isomorphic in codimension $1$ and $\rho(Y)=\rho(X')=\rho(X)-1$. 
 \end{proof}
\end{lem}

\subsection{Two birational models of a $\mathbb{Q}$-factorial terminal weak Fano $3$-fold}

In the following proposition, for a $\mathbb{Q}$-factorial terminal weak Fano $3$-fold, we construct $2$ special birational models with many nice geometric restrictions. This is a refinement of \cite[Proposition~4.1]{JZ23} based on ideas in \cite[Proposition~3.9]{CJ20}.
\begin{prop}\label{prop MFS}
Let $X$ be a $\mathbb{Q}$-factorial terminal weak Fano $3$-fold. 
Then there is a diagram
\[\begin{tikzcd}
	X & Y & Z \\
	& S
	\arrow["\phi", dashed, from=1-1, to=1-2]
	\arrow["\psi", dashed, from=1-2, to=1-3]
	\arrow["f", from=1-2, to=2-2]
\end{tikzcd}\]
satisfying the following properties:
\begin{enumerate}
 \item\label{prop MFS part1} $\phi$ is a birational contraction, $\psi$ is isomorphic in codimension $1$;
 \item\label{prop MFS part2} $Y$ is $\mathbb{Q}$-factorial and terminal, $-K_Y$ is big and $\mathbb{Q}$-movable;
 \item $f: Y\to S$ is a projective morphism with $f_*\mathcal{O}_Y=\mathcal{O}_S$ satisfying one of the following properties:
 \begin{enumerate}[label=(\roman*)]
 \item $S$ is a point and $\rho(Y)=1$; 
 
 \item $S\simeq \mathbb{P}^1$; 

 \item 
 $S$ is a del Pezzo surface with at worst Du Val singularities and $\rho(S)=1$.

 \end{enumerate}
 
 \item\label{prop MFS part3} $Z$ is a $\mathbb{Q}$-factorial canonical weak Fano $3$-fold.

\end{enumerate}

\end{prop}

\begin{proof}
{\bf Step 1}. We describe the construction of $f: Y\to S$. This step is the same as \cite[Proposition~4.1]{JZ23}.

We can run a $K_X$-MMP on $X$ which ends up with a Mori fiber space $Y\to S_1$ where $Y$ is $\mathbb{Q}$-factorial and terminal. Recall that by \cite[Theorem~1]{ZQ06}, $X$ is rationally connected, which implies that $Y$
and $S_1$ are also rationally connected.

If $\dim S_1\in \{0,1\}$, then 
we get Case (i) or Case (ii) respectively by taking $S=S_1$.

Now suppose that $\dim S_1=2$. By \cite[Theorem~1.2.7]{MP}, $S_1$ has at worst Du Val singularities. Hence 
we can run a $K_{S_1}$-MMP on $S_1$ which ends up with a surface $S_0$ with at worst Du Val singularities where either $S_0$ is a del Pezzo surface with $\rho(S_0)=1$ or $S_0$ is a surface with a Mori fiber space structure $S_0\to \mathbb{P}^1$. Then we get Case (iii) or Case (ii) by taking $S=S_0$ or $\mathbb{P}^1$ respectively.


Denote by $\phi: X\dashrightarrow Y$ the natural birational map.
By Lemma~\ref{lem X to X' bir cont}, $Y$ is of Fano type and $-K_Y$ is big and $\mathbb{Q}$-movable. 

\medskip

{\bf Step 2}. We describe the construction of $Z$.

Since $-K_Y$ is big and $Y$ is a Mori dream space, we may run a $(-K_Y)$-MMP on $Y$ which ends up with a variety $Z$ such that $Z$ is $\mathbb{Q}$-factorial and $-K_Z$ is nef and big. By Lemma~\ref{lem X to X' bir cont} for $X\dashrightarrow Z$, $Z$ is canonical. Hence $Z$ is a $\mathbb{Q}$-factorial canonical weak Fano $3$-fold.

Denote by $\psi: Y\dashrightarrow Z$ the natural birational map.
Since $-K_Y$ is $\mathbb{Q}$-movable, the $(-K_Y)$-MMP on $Y$ does not contract any divisor. So $\psi$ is isomorphic in codimension $1$. 
\end{proof}

\section{Geometry of weak Fano $3$-folds with $\rho\geq 2$}\label{sec 4}
In order to study the special birational models in Proposition~\ref{prop MFS}, it is crucial to study the geometry of $\mathbb{Q}$-factorial canonical weak Fano $3$-folds with $\rho\geq 2$. 
The following is the main theorem of this section.

\begin{thm}\label{thm picard 23}
Let $Z$ be a $\mathbb{Q}$-factorial canonical weak Fano $3$-fold with $\rho(Z)\geq 2$. Let $F$ be a movable prime divisor on $Z$. Then one of the following assertions holds:
 \begin{enumerate}
 \item there is a birational contraction $Z\dashrightarrow Z_0$ such that $Z_0$ is a $\mathbb{Q}$-factorial canonical Fano $3$-fold with $\rho(Z_0)=1$; 
 \item $(-K_Z)^3\leq64;$
 \item $-K_Z-\frac{2}{\irr(F)}F$ is not big. 
 \end{enumerate}
\end{thm}

Here recall that for a variety $F$, the {\it degree of irrationality} of $F$ is defined to be 
\[
\irr(F):=\min\left\{d\in \mathbb{Z}_{>0} \, \left| \, \begin{array}{c}\text{there exists a dominant rational map }\\ F\dashrightarrow \mathbb{P}^{\dim F} \text{ of degree } d\end{array}\right.\right\}.
\]
It is clear by definition that $F$ is a rational variety if and only if $\irr(F)=1$.

The following lemma tells when Theorem~\ref{thm picard 23}(3) holds. 

\begin{lem}\label{lem conic bundle w<2}
 Let $Z$ be a $\mathbb{Q}$-factorial klt weak Fano $3$-fold and let $F$ be a prime divisor on $Z$. Suppose that $Z'\to S$ is a Mori fiber space with $\dim S=2$ and
 there is a birational contraction $Z\dashrightarrow Z'$.
If $F$ dominates $S$, then $-K_Z-\frac{2}{\irr(F)}F$ is not big. 
\end{lem}
\begin{proof}
 By \cite[Theorem~1]{ZQ06}, $Z$ is rationally connected, which implies that $Z'$
and $S$ are rationally connected. In particular, $S$ is a rational surface. Denote by $C\simeq \mathbb{P}^1$ a general fiber of $Z'\to S$ and
denote by $F'$ the strict transform of $F$ on $Z'$, then the induced map $F\dashrightarrow S$ is a dominant rational map of degree $(F'\cdot C)$, which implies that $(F'\cdot C)\geq \irr(F)$. 
If $-K_Z-wF$ is big, then $-K_{Z'}-wF'$ is big and hence $-K_C-wF'|_C$ is big. This means that 
\[w<\frac{2}{(F'\cdot C)}\leq \frac{2}{\irr(F)}.\qedhere\] 
\end{proof}

The idea of the proof of Theorem~\ref{thm picard 23} is to look at Mori fiber spaces of $Z$. If all Mori fiber spaces of $Z$ are conic bundles, then we can apply Lemma~\ref{lem conic bundle w<2}; if some Mori fiber space of $Z$ is a del Pezzo fibration, then we can use the computations in \cite{Lai21} or \cite{Jiang21} to bound $(-K_Z)^3$.

\begin{proof}[Proof of Theorem~\ref{thm picard 23}]
By \cite[Theorem~1]{ZQ06}, $Z$ is rationally connected, so every variety dominated by $Z$ is rationally connected.

\medskip

 {\bf Step 1}. We consider the case $\rho(Z)=2$.

 If $Z$ has a $\mathbb{Q}$-fixed prime divisor, then by Lemma~\ref{lem contract fixed}, there exists a birational contraction $Z\dashrightarrow Z_0$ such that $Z_0$ is a $\bQ$-factorial canonical Fano $3$-fold with $\rho(Z_0)=1$. In this case, we get Assertion (1).

From now on, suppose that $Z$ has no $\mathbb{Q}$-fixed prime divisor. Then we can run a $(-F)$-MMP on $Z$, which ends up with a Mori fiber space $g_1: Z_1\to T_1$ such that $\varphi_{1*}F$ is ample over $T_1$, where $\varphi_1: Z\dashrightarrow Z_1$ is the natural map which is isomorphic in codimension $1$ by Lemma~\ref{lem exc is fixed}.

 If $\dim T_1=2$, then $-K_Z-\frac{2}{\irr(F)}F$ is not big by Lemma~\ref{lem conic bundle w<2} and we get Assertion (3).

From now on, suppose that $\dim T_1=1$, then $T_1\simeq \mathbb{P}^1$. 
Denote by $G_1$ a general fiber of $g_1$ and denote by $G$ the strict transform of $G_1$ on $Z$. 
By Lemma~\ref{lem X to X' bir cont}, $Z_1$ is canonical and $-K_{Z_1}$ is big and $\mathbb{Q}$-movable, hence $G_1$ is a weak del Pezzo surface with at worst Du Val singularities.

We can run a $(-G)$-MMP on $Z$, which ends up with a Mori fiber space $g_2: Z_2\to T_2$ such that $\varphi_{2*}G$ is ample over $T_2$, where $\varphi_{2}: Z\dashrightarrow Z_2$ is the natural map which is isomorphic in codimension $1$ by Lemma~\ref{lem exc is fixed}. 

If $\dim T_2=2$, then as $\varphi_{2*}G$ dominates $T_2$, $-K_Z-2G$ is not big by Lemma~\ref{lem conic bundle w<2}. Hence $-K_{Z_1}-2G_{1}$ is not big as $\varphi_1$ is isomoprhic in codimension $1$. 
So by \cite[Lemma~2.5]{Jiang18} and Lemma~\ref{lem DP surface}, 
\begin{align*}
(-K_Z)^3&{}=\Vol(Z_1, -K_{Z_1})\\
 &{}\leq\Vol(-K_{Z_1}-2G_1)+6\Vol(G_1,-K_{G_1})\\
 &{}=6(-K_{G_1})^2\leq54. 
\end{align*}
In this case, we get Assertion (2).

From now on, suppose that $\dim T_2=1$, that is, $T_2\simeq \mathbb{P}^1$. Denote by $G_2$ a general fiber of $g_2$ and denote by $G'$ the strict transform of $G_2$ on $Z$. 
By Lemma~\ref{lem X to X' bir cont}, $Z_2$ is canonical and $-K_{Z_2}$ is big and $\mathbb{Q}$-movable, hence $G_2$ is a weak del Pezzo surface with at worst Du Val singularities.

By construction, $G$ and $G'$ are not big, which means that they generate the cone of effective divisors of $Z$ as $\rho(Z)=2$. 
We may write \[-K_Z\equiv a_1G+a_2G'\] for some positive rational numbers $a_1, a_2$. Take $W$ to be a common resolution of $Z, Z_1, Z_2$ with natural maps $p:W\to Z$, $q_i:W\to Z_i\, (i=1,2)$. 
For $i=1,2$, take $H_i$ to be the strict transform of $G_i$ on $W$, then $H_i$ is a general fiber of $g_i\circ q_i$ by construction.

\[\begin{tikzcd}
	& {H_1} & W & {H_2} \\
	{\varphi_{1*}G=G_1} & {Z_1} & Z & {Z_2} & {G_2=\varphi_{2*}G'} \\
	& {\mathbb{P}^1} && {\mathbb{P}^1}
	\arrow["\subset"{description}, draw=none, from=1-2, to=1-3]
	\arrow[from=1-2, to=2-1]
	\arrow["{q_1}"', from=1-3, to=2-2]
	\arrow["p", from=1-3, to=2-3]
	\arrow["{q_2}", from=1-3, to=2-4]
	\arrow["\supset"{description}, draw=none, from=1-4, to=1-3]
	\arrow[from=1-4, to=2-5]
	\arrow["\subset"{description}, draw=none, from=2-1, to=2-2]
	\arrow["{g_1}"', from=2-2, to=3-2]
	\arrow["{\varphi_1}"', dashed, from=2-3, to=2-2]
	\arrow["{\varphi_2}", dashed, from=2-3, to=2-4]
	\arrow["{g_2}", from=2-4, to=3-4]
	\arrow["\supset"{description}, draw=none, from=2-5, to=2-4]
\end{tikzcd}\]

\begin{claim}\label{coeff}
 \begin{enumerate}
 \item $(p^*(-K_Z)|_{H_i})^2\leq K_{G_i}^2$ for $i=1, 2$;
 \item $a_1 K_{G_1}^2+a_2K_{G_2}^2\leq 64$.
 \end{enumerate}
\end{claim}
\begin{proof}
(1) Fix $i\in\{1,2\}$.
 By the negativity lemma (\cite[Lemma~3.39]{KM}), $q_i^{*}(-K_{Z_i})-p^*(-K_Z)\geq 0$ is an effective $q_i$-exceptional $\mathbb{Q}$-divisor on $W$. 
Then $q_i^{*}(-K_{Z_i})|_{H_i}-p^*(-K_Z)|_{H_i}\geq 0$
and $q_i^*(-K_{Z_i})|_{H_i}=(q_i|_{H_i})^*(-K_{G_i})$ is nef. Hence $(p^*(-K_Z)|_{H_i})^2\leq (q_i^*(-K_{Z_i})|_{H_i})^2= K_{G_i}^2$. 

(2) By taking pushforward to $Z_1$, we get $-K_{Z_1}\equiv a_1G_1+a_2\varphi_{1*}G'$, which implies that $-K_{G_1}\equiv a_2\varphi_{1*}G'|_{G_1}$. Note that $\varphi_{1*}G'|_{G_1}$ is a movable Weil divisor on $G_1$, so 
by Lemma~\ref{lem DP surface}, \[
a_2\leq \begin{cases} 4& \text{if } K_{G_1}^2\leq 8;\\
3& \text{if } K_{G_1}^2=9.
\end{cases}
\]
By the same argument, we have 
\[
a_1\leq \begin{cases} 4& \text{if } K_{G_2}^2\leq 8;\\
3& \text{if } K_{G_2}^2=9.
\end{cases}
\]
Then it follows that $a_1 K_{G_1}^2+a_2K_{G_2}^2\leq 64$.
\end{proof}

Then by the projection formula and Claim~\ref{coeff},
\begin{align*}
 (-K_Z)^3={}&(p^*(-K_Z))^2\cdot (a_1H_1+a_2H_2)\\
 ={}&a_1(p^*(-K_Z)|_{H_1})^2+a_2(p^*(-K_Z)|_{H_2})^2 \\ 
 \leq {}&a_1 K_{G_1}^2+a_2K_{G_2}^2\leq 64. 
\end{align*}
 In this case, we get Assertion (2).

\medskip

 {\bf Step 2}. We consider the case $\rho(Z)\geq 3$.

If $Z$ has a $\mathbb{Q}$-fixed prime divisor, then by Lemma~\ref{lem contract fixed}, there exists a birational contraction $Z\dashrightarrow Z'$ such that $Z'$ is a $\bQ$-factorial canonical weak Fano $3$-fold with $\rho(Z')=\rho(Z)-1$. Denote by $F'$ the strict transform of $F$ on $Z'$, then $F'$ is a movable prime divisor with $\irr(F')=\irr(F)$. 
By Lemma~\ref{lem X to X' bir cont}, $(-K_Z)^3\leq (-K_{Z'})^3$. 
Note that if $-K_{Z'}-\frac{2}{\irr(F')}F'$ is not big, then
$-K_Z-\frac{2}{\irr(F)}F$ is not big. 
So if one of Assertions (1)--(3) holds for $Z'$, then it also holds for $Z$.
So we can conclude the proof by induction on $\rho(Z)$ and Step 1.

If $Z$ has no $\mathbb{Q}$-fixed prime divisor, then we can run a $(-F)$-MMP on $Z$, which ends up with a Mori fiber space $Z''\to S'$ such that $F''$ is ample over $S'$, where $F''$ is the strict transform of $F$ on $Z''$. 
By Lemma~\ref{lem exc is fixed}, $\rho(Z'')=\rho(Z)$.
Then $\rho(S')=\rho(Z'')-1\geq 2$, which implies that $\dim S'=2$ and hence $-K_Z-\frac{2}{\irr(F)}F$ is not big by Lemma~\ref{lem conic bundle w<2}.
\end{proof}

\section{Proof of the main theorem}\label{sec 5}
In this section, we prove Theorem~\ref{mainthm}.
As a preparation, we use Theorem~\ref{thm picard 23} to treat the models in Proposition~\ref{prop MFS}.
 
\begin{prop}\label{prop final}
Keep the same notation as in Proposition~\ref{prop MFS}. 
Then one of the following assertions holds:
\begin{enumerate}

 \item there is a birational contraction $Z\dashrightarrow Z_0$ such that $Z_0$ is a $\mathbb{Q}$-factorial canonical Fano $3$-fold with $\rho(Z_0)=1$;
 \item $(-K_Z)^3\leq64$.
\end{enumerate}

\end{prop}

The idea is to construct a basepoint-free divisor $G$ on the model $Y$ in Proposition~\ref{prop MFS}. If $-K_Y-\frac{2}{\irr(G)}G$ is big, then $-K_Z-\frac{2}{\irr(G)}\psi_*G$ is big and we get Assertion (1) or (2) by Theorem~\ref{thm picard 23}; if $-K_Y-\frac{2}{\irr(G)}G$ is not big, then we can give a good estimate of $(-K_Z)^3$ by volumes of divisors on $G$. 
\begin{proof}
Keep the same notation as in Proposition~\ref{prop MFS}.

We discuss by the property of $Y\to S$. Recall that $\psi:Y\dashrightarrow Z$ is isomorphic in codimension $1$ and hence $(-K_Z)^3=\Vol(Y, -K_Y)$.

If $Y\to S$ satisfies (i) in Proposition~\ref{prop MFS}, then 
$Y$ is a $\mathbb{Q}$-factorial terminal Fano $3$-fold with $\rho(Y)=1$. In this case, $Z\simeq Y$ and $(-K_Z)^3\leq 64$ by \cite{Nami} and \cite[Theorem~1.2]{Prok07}.

If $Y\to S$ satisfies (ii) in Proposition~\ref{prop MFS}, then a general fiber $G$ of $Y\to \mathbb{P}^1$ is a smooth weak del Pezzo surface as $Y$ is terminal and $-K_Y$ is big and $\mathbb{Q}$-movable. Then by Lemma~\ref{lem DP surface}, $K_G^2\leq 9$. 
If $-K_Y-2G$ is not big, then by \cite[Lemma~2.5]{Jiang18}, 
\begin{align*}
 (-K_Z)^3=\Vol(Y, -K_Y)\leq\Vol(Y, -K_Y-2G)+6\Vol(G,-K_G)=6K_G^2\leq54.
\end{align*} 
If $-K_Y-2G$ is big, then $-K_Z-2\psi_*G$ is big and we get Assertion (1) or (2) by Theorem~\ref{thm picard 23}.

From now on, suppose that $Y\to S$ satisfies (iii) in Proposition~\ref{prop MFS}.

We construct a basepoint-free divisor $G$ on $Y$ as in \cite[Proof of Proposition~4.3]{JZ23}.
Take $H$ to be a general member of the following basepoint-free linear system on $S$: 
\[ 
\begin{cases}
|\mathcal{O}_{\mathbb{P}^2}(1)| & \text{if } S\simeq \mathbb{P}^2;\\
|\mathcal{O}_{\mathbb{P}(1, 1, 2)}(2)| & \text{if } S\simeq \mathbb{P}(1, 1, 2);\\
|-K_{S}|& \text{if } 2\leq K_{S}^2\leq 6;\\
|-2K_{S}|& \text{if } K_{S}^2=1.
\end{cases}
\]
These cases cover all possible $S$ by the classification of del Pezzo surfaces with Du Val singularities and Picard rank $1$ (see \cite{MZ88}, \cite[Remark~3.4(ii)]{Prok07}).

Take $G:=f^{-1}(H)=f^*H$ which is a basepoint-free divisor on $Y$.
Then both $H$ and $G$ are smooth by Bertini's theorem. 
Denote by $F$ a general fiber of $G\to H$ which is a smooth rational curve and denote
\[m:=(H^2)=
\begin{cases}
1 & \text{if } S\simeq \mathbb{P}^2;\\
2 & \text{if } S\simeq \mathbb{P}(1, 1, 2);\\
K_S^2 & \text{if } 2\leq K_{S}^2\leq 6;\\
4 & \text{if } K_{S}^2=1.
\end{cases}
\]
Then $G|_G\equiv (H^2)\cdot F=mF$.
We know that \[-K_Y|_G=-K_G+G|_G\equiv -K_G+mF\] is nef as $-K_Y$ is $\mathbb{Q}$-movable.

\medskip

{\bf Case 1}: $K_{S}^2=1$.

In this case, by \cite[Proof of Proposition~4.3, Page 9313]{JZ23},
\[\Vol(G, -K_{Y}|_G)\leq 4(-K_S\cdot H)\leq8.\]
If $-K_Y-2G$ is not big, then by \cite[Lemma~2.5]{Jiang18}, 
\begin{align*}
(-K_Z)^3= \Vol(Y, -K_Y)\leq\Vol(Y, -K_Y-2G)+6\Vol(G,-K_Y|_G)\leq48.
\end{align*}
If $-K_Y-2G$ is big, then $-K_Z-2\psi_*G$ is big and we get Assertion (1) or (2) by Theorem~\ref{thm picard 23}.

\medskip

{\bf Case 2}: $S\simeq \mathbb{P}^2$ or $S\simeq \mathbb{P}(1, 1, 2)$.

In this case, it is clear that $H\simeq \mathbb{P}^1$ by the genus formula.
Then $G\to H\simeq \mathbb{P}^1$ is factored through by $\mathbb{F}_n$ for some $n\in \mathbb{Z}_{\geq0}$.
Keep Notation~\ref{notation Fn En}, where $\ell$ is the ruling of $\mathbb{F}_n\to H\simeq \mathbb{P}^1$. 

We claim that $n\leq m+2$.
 By taking pushforward of $-K_G+mF$ to $\mathbb{F}_n$, we get that \[-K_{\mathbb{F}_n}+m\ell\equiv 2\sigma_0+(n+2+m)\ell\] is nef. Hence by Lemma~\ref{lem Fn En divisors}, we have $n\leq m+2$.

By Theorem~\ref{thm picard 23}, we may assume that $-K_Y-2G$ is not big. Then by Lemma~\ref{lem vol inequality}, 
\begin{align*}
 \Vol(Y,-K_Y) &{}\leq\Vol(Y,-K_Y-2G)+3\int_0^2 \Vol(G,(-K_Y-tG)|_G)\,\mathrm{d}t\\
 &{}=3\int_{0}^2 \Vol(G,-K_G+(m-mt)F)\,\mathrm{d}t\\
 &{}\leq 3\int_{0}^2 \Vol(\mathbb{F}_n,-K_{\mathbb{F}_n}+(m-mt)\ell)\,\mathrm{d}t\\
 &{}=3\int_{0}^2 \Vol(\mathbb{F}_n,2\sigma_0+(n+2+m-mt)\ell)\,\mathrm{d}t\\
 &{}=\frac{3}{m}\int_{n+2-m}^{n+2+m} \Vol(\mathbb{F}_n,2\sigma_0+t\ell)\,\mathrm{d}t.
\end{align*}

If $n=0$, then by Lemma~\ref{lem Fn En divisors},
\begin{align*}
 \Vol(Y,-K_Y)&{}\leq \frac{3}{m}\int_{2-m}^{2+m} \Vol(\mathbb{F}_0,2\sigma_0+t\ell)\,\mathrm{d}t\\
 &{}=\frac{3}{m}\int_{2-m}^{2+m} 4t\,\mathrm{d}t\\
 &{}=48.
\end{align*}

If $n>0$, then by Lemma~\ref{lem Fn En divisors},
\begin{align*}
 \Vol(Y,-K_Y) &{}\leq \frac{3}{m}\int_{n+2-m}^{n+2+m} \Vol(\mathbb{F}_n,2\sigma_0+t\ell)\,\mathrm{d}t\\
 &{}=\frac{3}{m}\int_{n+2-m}^{2n} \frac{t^2}{n}\,\mathrm{d}t+\frac{3}{m}\int_{2n}^{n+2+m} 4(t-n)\,\mathrm{d}t\\
 &{}=\begin{cases}
\frac{(n-1)^3}{n}+48 & \text{if } m=1;\\
\frac{n^2}{2}+48 & \text{if } m=2
\end{cases}\\
&{}\leq 56.
\end{align*}

In summary, $(-K_Z)^3=\Vol(Y, -K_Y)\leq 56$ in this case, assuming that $-K_Y-2G$ is not big. 

\medskip

{\bf Case 3}: $2\leq K_{S}^2\leq 6.$

In this case, $H\in |-K_S|$ is an elliptic curve by the genus formula. 
Then $G\to H$ is factored through by some $\mathbb{E}_n$ for some $n\in \mathbb{Z}_{\geq -1}$ (see \cite[Theorem~V.2.12, Theorem~V.2.15]{H}). 
Keep Notation~\ref{notation Fn En}.

We claim that $n\leq m$.
By taking pushforward of $-K_G+mF$ to $\mathbb{E}_n$, we see that \[-K_{\mathbb{E}_n}+m\ell\equiv 2\sigma_0+(n+m)\ell\] is nef. By 
Lemma~\ref{lem Fn En divisors} and Lemma~\ref{lem E-1 divisors}, 
$n\leq m$.

By Theorem~\ref{thm picard 23}, we may assume that $-K_Y-G$ is not big as $G$ is irrational. Then by Lemma~\ref{lem vol inequality}, 
\begin{align*}
 \Vol(Y,-K_Y)&{}\leq\Vol(Y,-K_Y-G)+3\int_0^1 \Vol(G,(-K_Y-tG)|_G)\,\mathrm{d}t\\
 &{}=3\int_{0}^1 \Vol(G,-K_G+(m-mt)F)\,\mathrm{d}t\\
 &{}\leq 3\int_{0}^1 \Vol(\mathbb{E}_n,-K_{\mathbb{E}_n}+(m-mt)\ell)\,\mathrm{d}t\\
 &{}=3\int_{0}^1 \Vol(\mathbb{E}_n,2\sigma_0+(n+m-mt)\ell)\,\mathrm{d}t\\&{}=\frac{3}{m}\int_{n}^{n+m} \Vol(\mathbb{E}_n,2\sigma_0+t\ell)\,\mathrm{d}t. 
\end{align*}

If $-1\leq n\leq0$, then by 
Lemma~\ref{lem Fn En divisors} and Lemma~\ref{lem E-1 divisors}, $2\sigma_0+t\ell$ is nef for $ n\leq t\leq n+m$. So 
\begin{align*}
 \Vol(Y,-K_Y)&{}\leq\frac{3}{m}\int_{n}^{n+m} \Vol(\mathbb{E}_n,2\sigma_0+t\ell)\,\mathrm{d}t\\
 &{}=\frac{3}{m}\int_{n}^{n+m} (4t-4n)\,\mathrm{d}t\\
 &{}=6m\leq 36.
\end{align*}

Finally, suppose that $1\leq n\leq m$. By Lemma~\ref{lem Fn En divisors}, 
\begin{align*}
 \Vol(Y,-K_Y)&{}\leq \frac{3}{m}\int_{n}^{n+m} \Vol(\mathbb{E}_n,2\sigma_0+t\ell)\,\mathrm{d}t \\
 {}&=\frac{3}{m}\int_{n}^{2n} \frac{t^2}{n}\,\mathrm{d}t +\frac{3}{m}\int_{2n}^{n+m} 4(t-n)\,\mathrm{d}t \\
 &{}=\frac{n^2}{m}+6m\leq 7m\leq 42.
\end{align*}

In summary, $(-K_Z)^3=\Vol(Y,-K_Y)\leq 42$ in this case, assuming that $-K_Y-G$ is not big.
 
Then the proof is completed. 
\end{proof}

\begin{proof}[Proof of Theorem~\ref{mainthm}]
Let $X$ be a canonical weak Fano $3$-fold. By \cite[Theorem~6.23, Theorem~6.25]{KM}, we can take a $\bQ$-factorial terminalization $\pi: X'\longrightarrow X$ such that $K_{X'}=\pi^*K_X$ and $X'$ is a $\mathbb{Q}$-factorial terminal weak Fano $3$-fold.
In particular, $(-K_X)^3=(-K_{X'})^3$ and $h^0(X, -mK_X)=h^0(X', -mK_{X'})$ for any positive integer $m$. 
 
Let $Z$ be the birational model of $X'$ constructed by Proposition~\ref{prop MFS}.
Then by Proposition~\ref{prop final}, either there is a birational contraction $Z\dashrightarrow Z_0$ such that $Z_0$ is a $\mathbb{Q}$-factorial canonical Fano $3$-fold with $\rho(Z_0)=1$ or $(-K_Z)^3\leq 64$.

 In the former case, by Lemma~\ref{lem X to X' bir cont} for $X'\dashrightarrow Z_0$, we have $(-K_{X'})^3\leq(-K_{Z_0})^3\leq72$ by \cite[Theorem~1.1]{JLL25}. 
 In the latter case, by Lemma~\ref{lem X to X' bir cont} for $X'\dashrightarrow Z$, we have $(-K_{X'})^3\leq(-K_{Z})^3\leq 64$.
 Hence in both cases, $(-K_X)^3=(-K_{X'})^3\leq 72$. 
This proves Assertion (1).

By the above discussion, we actually showed that if $(-K_X)^3>64$, then there is a birational contraction $X'\dashrightarrow Z_0$ such that $Z_0$ is a $\mathbb{Q}$-factorial canonical Fano $3$-fold with $\rho(Z_0)=1$. We rephrase this fact in Remark~\ref{rem vol>64}.
 
Following Reid's Riemann--Roch formula (\cite[2.2(3)]{CJ16}), we have
\begin{align*}
h^0(X', -mK_{X'})\leq {}& \frac{1}{12}m(m+1)(2m+1)(-K_{X'})^3+2m+1\\
\leq {}& (2m+1)(6m^2+6m+1),
\end{align*}
and the equality holds only if $(-K_{X'})^3=72$. Thus we get Assertion (2).

Finally, we discuss the equality case. 

Suppose that one of the inequalities in Assertions (1) and (2) holds, then we always have $(-K_X)^3=72$. Moreover, there is a birational contraction $X'\dashrightarrow Z_0$ such that $Z_0$ is a $\mathbb{Q}$-factorial canonical Fano $3$-fold with $\rho(Z_0)=1$ and $(-K_{X'})^3=(-K_{Z_0})^3=72.$ Then $Z_0\simeq \mathbb{P}(1,1,1,3)$ or $\mathbb{P}(1,1,4,6)$ by \cite[Theorem~1.1]{JLL25}. Take a commen resolution $W$ of $X'$ and $Z_0$ with morphisms $p: W\to X'$ and $q: W\to Z_0$, 
then we have
\[p^*(-K_{X'})=q^*(-K_{Z_0})-E\]
where $E\geq0$ is an effective $q$-exceptional $\mathbb{Q}$-divisor on $W$ by the negativity lemma (\cite[Lemma~3.39]{KM}). 
Since both $p^*(-K_{X'})$ and $q^*(-K_{Z_0})$ are nef and 
$(-K_{X'})^3=(-K_{Z_0})^3=72,$ we get $E=0$ by \cite[Theorem~A]{FKL16}. 
In particular, 
\[p^*\pi^*(-K_{X})=p^*(-K_{X'})=q^*(-K_{Z_0}).\]
As $-K_{Z_0}$ is ample, there is a natural birational morphism $g: X\to Z_0$ with $g^*K_{Z_0}=K_{X}$, which is just the anti-canonical model of $X$. We just take $X_0=Z_0$ to finish the proof.
Furthermore, if $-K_X$ is ample, then $g$ is an isomorphism and $X\simeq X_0$.

Conversely, if there is a birational morphism $g: X\to X_0$ where $X_0\simeq \mathbb{P}(1,1,1,3)$ or $\mathbb{P}(1,1,4,6)$ and $g^*K_{X_0}=K_X$, then it is a direct computation that all inequalities in Assertions (1) and (2) are equalities. Here we can use the usual Riemann--Roch formula as $X_0$ is Gorenstein. 
\end{proof}

\section*{Acknowledgments} 
The authors would like to thank Haidong Liu and Jie Liu for helpful discussions. This work was supported by National Key Research and Development Program of China (No. 2023YFA1010600, No. 2020YFA0713200) and NSFC for Innovative Research Groups (No. 12121001). C.~Jiang is a member of the Key Laboratory of Mathematics for Nonlinear Sciences, Fudan University.


\begin{thebibliography}{99}

\bibitem{Bir16b} C.~Birkar, \emph{Singularities of linear systems and boundedness of Fano varieties}, Ann. of Math. (2) 193 (2021), no. 2, 347--405.

\bibitem{BCHM} C.~Birkar, P.~Cascini, C.~D.~Hacon, J.~M\textsuperscript{c}Kernan, \emph{Existence
of minimal models for varieties of log general type}, J. Amer. Math. Soc. 23 (2010), no. 2, 405--468.

\bibitem{BFJ09} S.~Boucksom, C.~Favre, M.~Jonsson, \emph{Differentiability of volumes of divisors and a problem of Teissier}, J. Algebraic Geom. 18 (2009), no. 2, 279--308.








\bibitem{CJ16} M.~Chen, C.~Jiang, {\em On the anti-canonical geometry of $\bQ$-Fano threefolds}, J. Differential Geom. {104} (2016), no. 1, 59--109.

\bibitem{CJ20} M.~Chen, C.~Jiang, \emph{On the anti-canonical geometry of weak $\bQ$-Fano threefolds, II}, Ann. Inst. Fourier (Grenoble) 70 (2020), no. 6, 2473--2542.

\bibitem{Dol12} I.~V.~Dolgachev, \emph{Classical algebraic geometry: A modern view}, Cambridge University Press, Cambridge, 2012.


\bibitem{FKL16}M.~Fulger, J.~Koll\'ar, B.~Lehmann, {\it Volume and Hilbert function of $\mathbb{R}$-divisors}, Michigan Math. J. 65 (2016), no. 2, 371--387.
 

\bibitem{H} R.~Hartshorne, \emph{Algebraic geometry}, Graduate Texts in Mathematics, No. 52. Springer-Verlag, New York-Heidelberg, 1977.

\bibitem{HK00} Y.~Hu, S.~Keel. {\em Mori dream spaces and GIT}, Michigan Math. J. 48 (2000), 331--348. 

\bibitem{IP99} V.~A.~Iskovskikh, Y.~G.~Prokhorov, {\it Fano varieties}, Algebraic geometry, V, Encyclopaedia Math. Sci., 47, Springer, Berlin, 1999.

 

\bibitem{Jiang18} C.~Jiang, \emph{On birational boundedness of Fano fibrations}, Amer. J. Math. 140 (2018), no. 5, 1253--1276.


\bibitem{Jiang21} C.~Jiang, \emph{Boundedness of anti-canonical volumes of singular log Fano threefolds}, Comm. Anal. Geom. 29 (2021), no. 7, 1571--1596.

\bibitem{JLL25} C.~Jiang, H.~Liu, J.~Liu, \emph{Optimal upper bound for degrees of canonical Fano threefolds of Picard number one}, arXiv:2501.16632v1.


\bibitem{JZ23} C.~Jiang, Y.~Zou, \emph{An effective upper bound for anti-canonical volumes of canonical $\mathbb{Q}$-Fano threefolds}, Int. Math. Res. Not. IMRN (2023), no. 11, 9298--9318.

\bibitem{JZ24} C.~Jiang, Y.~Zou, \emph{On the anti-canonical geometry of weak $\mathbb{Q}$-Fano threefolds, III}, Nagoya Math. J. 253 (2024), 23--47. 
 

\bibitem{Kaw92} Y.~Kawamata, {\em Boundedness of $\bQ$-Fano threefolds}, Proceedings of the International Conference on Algebra, Part 3 (Novosibirsk, 1989), pp. 439--445, Contemp. Math., 131, Part 3, Amer. Math. Soc., Providence, RI, 1992. 



\bibitem{KMMT00} J.~Koll\'ar, Y.~Miyaoka, S.~Mori, H.~Takagi, {\em
Boundedness of canonical $\bQ$-Fano $3$-folds}, Proc. Japan Acad. Ser. A Math. Sci. 76 (2000), no. 5, 73--77.

\bibitem{KM} J.~Koll\'{a}r, S.~Mori, \emph{Birational geometry of algebraic varieties},
Cambridge tracts in mathematics, 134, Cambridge University
Press, Cambridge, 1998.
 

\bibitem{Lai21} C.-J.~Lai, \emph{On anticanonical volumes of weak $\bQ$-Fano terminal threefolds of Picard rank two}, Ann. Sc. Norm. Super. Pisa Cl. Sci. (5) 22 (2021), no. 1, 315--331. 


\bibitem{LaiLee} C.-J.~Lai, T.-J.~Lee, \emph{A supplement to the anticanonical volumes of weak $\bQ$-Fano terminal threefolds of Picard rank two}, arXiv:2502.19419v1. 

\bibitem{Positivity1} R.~Lazarsfeld, \emph{Positivity in algebraic geometry, I, Classical setting: line bundles and linear series}, Results in Mathematics and Related Areas. 3rd Series. A Series of Modern Surveys in Mathematics, 48. Springer-Verlag, Berlin, 2004. 

\bibitem{Positivity2} R.~Lazarsfeld, \emph{Positivity in algebraic geometry, II, Positivity for vector bundles, and multiplier ideals}, Results in Mathematics and Related Areas. 3rd Series. A Series of Modern Surveys in Mathematics, 49. Springer-Verlag, Berlin, 2004. 

\bibitem{LM09}
R.~Lazarsfeld, M.~Musta{\c{t}}{\u{a}}, {\it Convex bodies associated to linear series},
Ann. Sci. {\'E}c. Norm. Sup{\'e}r. (4) 42 (2009), no. 5, 783--835.


\bibitem{MZ88} M.~Miyanishi, D.-Q.~Zhang, \emph{Gorenstein log del Pezzo surfaces of rank one}, J. Algebra 118 (1988), no. 1, 63--84.

\bibitem{MP} S.~Mori, Y.~G.~Prokhorov, {\em On $\bQ$-conic bundles}, Publ. Res. Inst. Math. Sci. 44 (2008), no. 2, 315--369.

\bibitem{Nami} Y.~Namikawa, {\em Smoothing Fano $3$-folds}, J. Algebraic Geom. 6 (1997), no. 2, 307--324. 

\bibitem{Prok05}Y.~G.~Prokhorov,
{\em The degree of Fano threefolds with canonical Gorenstein singularities}, Mat. Sb. 196 (2005), no. 1, 81--122; translation in
Sb. Math. 196 (2005), no. 1-2, 77--114.

\bibitem{Prok07}Y.~G.~Prokhorov,
{\em The degree of $\bQ$-Fano threefolds}, Mat. Sb. 198 (2007), no. 11, 153--174; translation in
Sb. Math. 198 (2007), no. 11-12, 1683--1702.


 

\bibitem{PS09} Y.~G.~Prokhorov, V. V. Shokurov, {\em Towards the second main theorem on complements}, J. Algebraic Geom. 18 (2009), no. 1, 151–199.



\bibitem{ZQ06} Q.~Zhang, {\em Rational connectedness of log $\bQ$-Fano varieties}, J. Reine Angew. Math. {590} (2006), 131--142.

\end{thebibliography}
\end{document}